\documentclass[a4paper,12pt,reqno]{amsart}
\usepackage[T1]{fontenc} 
\usepackage[utf8]{inputenc} 
\usepackage[english]{babel} 
\usepackage{url} 
\usepackage{amsmath}
\usepackage{amssymb}
\usepackage{amsthm}
\usepackage{mathrsfs}
\usepackage{mathtools}
\usepackage[italian]{varioref}
\usepackage{hyperref}
\usepackage{typehtml}
\usepackage{multicol}
\usepackage{enumerate}
\usepackage{color}
\usepackage{mathrsfs}
\usepackage{comment}
\usepackage{mathabx}
\usepackage{tikz-cd}
\usepackage{adjustbox}
\usepackage{bbm}

\newcommand{\grad}{\nabla}

\newcommand{\R}{\mathbb R}

\newcommand{\e}{\mathrm{e}}                
\newcommand{\de}{\mathop{}\!\mathrm{d}}  
\newcommand{\pa}{\mathop{}\!\partial}
\newcommand{\opde}[1]{{\dfrac {\mathrm{d} \phantom{#1}}{\mathrm{d} {#1}}}}

\newcommand{\eps}{\varepsilon}

\newcommand{\fcr}{\mathcal X}
\newcommand{\Cal}[1]{{\mathcal {#1}}}

\newcommand{\debole}{\rightharpoonup}

\newcommand{\wass}{{\mathscr{W}}}
\newcommand{\disc}{{\mathscr{D}}}

\numberwithin{equation}{section}

\theoremstyle{definition}
\newtheorem{Definition}{Definition}[section]

\theoremstyle{plain}
\newtheorem{Theorem}{Theorem}[section]
\newtheorem{lemma}{Lemma}[section]

\newtheorem{Proposition}{Proposition}[section]
\theoremstyle{remark}

\title[Mean-field limit for topological interactions]
{Mean-field limit for particle systems with topological interactions}

\author[D.\ Benedetto]{Dario Benedetto}
\address{Dario Benedetto \hfill\break \indent
	Dipartimento di Matematica, Universit\`a di Roma `La Sapienza'
	\hfill\break \indent
	P.le Aldo Moro 2, 00185 Roma, Italy}
\email{benedetto@mat.uniroma1.it}

\author[E.\ Caglioti]{Emanuele Caglioti}
\address{Emanuele Caglioti \hfill\break \indent
	Dipartimento di Matematica, Universit\`a di Roma `La Sapienza'
	\hfill\break \indent
	P.le Aldo Moro 2, 00185 Roma, Italy}
\email{caglioti@mat.uniroma1.it}

\author[S.\ Rossi]{Stefano Rossi}
\address{Stefano Rossi\hfill\break \indent
	Dipartimento di Matematica, Universit\`a di Roma `La Sapienza'
	\hfill\break \indent
	P.le Aldo Moro 2, 00185 Roma, Italy}
\email{stef.rossi@uniroma1.it}

\begin{document}

\date{\today}
\begin{abstract}
  The mean-field limit for systems of self-propelled agents
  with ``topological interaction'' cannot be obtained by means of the
  usual Dobrushin approach.
  We get a result on this direction
  by adapting to the multidimensional case
  the techniques developed by Trocheris in 1986
  to treat the Vlasov-Poisson equation in one dimension.
\end{abstract}


\maketitle

\pagestyle{headings}

\section{Introduction}

Many interesting physical systems can be described 
at the microscopic level as particle dynamics 
and at the mesoscopic level with kinetic equations.
In the wide field of two-body interactions,
the link between these two regimes is mathematically well understood
in the case of the mean-field limit, 
i.e. when the density of the particles
diverges with their number $N$, the mean free path
vanishes as $1/N$ and the interaction intensity scales with $1/N$.
In this limit, each particle feels the interaction with the others
as a mean.

A rigorous mathematical proof of this result can be done 
in the case of two-body interactions with sufficiently regular 
potentials.
This classical achievement has been obtained
independently by several authors in the '$70$s 
(see \cite{BH,dobrushin,Neuzert}) 
and its explanation is particularly clear
in the Dobrushin's argument \cite{dobrushin}
where the result follows by noticing that
the empirical measure associated with the particle system
is a weak solution of the mean-field equation;
the proof follows by showing the weak continuity, w.r.t the initial datum,
of the weak solutions. 

Although the theory for regular pairwise interactions 
is sufficiently well understood, 
going beyond it considering singular potentials,
is instead a harder task.
This is the case of the three-dimensional
Vlasov-Poisson equation, which is
the most important equation of plasma physics and of galactic dynamics,
based on the choice of the Coulomb or Newton potential, respectively.
In this equation, the potential $1/r$ is singular at the origin and does not belong to any $L^p$ space.
Although the mean-field limit for the Vlasov-Poisson equation remains an open problem,
 there has been important progress in recent years,
see the works \cite{HJ1,HJ2} where the mean-field limit is proven for potentials with singularities ``weaker than $1/r$'' and also \cite{L,LP}. 
However, in the case of the one-dimensional Vlasov-Poisson equation, the problem
has been solved in \cite{trocheris, trocheris2}
and with a simpler proof in \cite{Hauray},
being the force discontinuous, but not diverging.

The mean-field limit is a case
of propagation of chaos, i.e. the $j$-particle distribution function
factorize in the limit. This property is the key for
obtaining a kinetic description of the particle dynamics 
(see for instance \cite{S} and \cite{Chaintron-Diez,G,J} for some reviews
on this point of view).

In recent years, the conceptual and mathematical apparatus of
kinetic equations 
has been used in the study of self-propelled particle systems
of 
biological nature, in particular for the motion of swarms
and other animals.
Starting with the pioneering paper in \cite{vicsek},
several models have been proposed to explain the
evolution of these systems.
In the simplest
\cite{CS1,CS2,vicsek}, a bird is modeled as a self-propelling particle that 
interacts with its neighbors. The interaction is such that neighboring birds
tend to align their velocities.
For many of these models, the mean-field limit has often been used to obtain
a kinetic description of the dynamics (see, for instance,
\cite{HaLiu, CFTV, CCJR, carrillo-wasserstein, GolseHa, bbc}).

A few years ago, supported by observational data
(\cite{top1, top2, Attanasi_et_al}),
``topological'' models for interaction were introduced:
an agent reacts to the presence of another not according to the distance, 
but according to the proximity ranking
(see eq.s \eqref{eq:CS}, \eqref{eq:pij}, \eqref{eq:Mpart} below
for a rigorous formulation).
These models come out of the case of two-body interaction,
and present various problems in their kinetic treatment.
In particular, the solutions of the kinetic equation
are not weakly continuous w.r.t. the initial datum
and there are also some
difficulties in defining the particle motion.

In this paper we prove a result on the mean-field limit for topological
models.
We focus our attention on the topological Cucker-Smale model,
but, with the same ideas, it is possible to consider more general
cases.
A first result in this direction has been proved in 
\cite{haskovec}, for a smoothed version of the model in which
the weak continuity in the initial datum is recovered.
We also mention that a kinetic Boltzmann equation for a stochastic particle
model with rank-based interaction has been obtained in
\cite{deg-pul}, by using the BBGKY hierarchy.

\vskip.10cm

We formulate the problem and summarize our results.
A Cucker-Smale type model for the motion of $N$ agents,
in the mean-field scaling, is the system
\begin{equation}
  \label{eq:CS}
  \left\{
  \begin{aligned}
    &\dot X_i(t)=V_i(t)\\
    &\dot V_i(t)= \frac 1N \sum_{j=1}^N p_{ij} (V_j(t) - V_i(t)),
  \end{aligned}
  \right.
\end{equation}
where the ``communication weights'' $\{p_{ij}\}_{i,j=1}^N$
are positive functions that take into account 
the interactions between agents. In classical models,
$p_{ij}$ depends only on the distance $|X_i - X_j|$
between the agents.
In topological models the weights depend on the positions of the agents
by their rank
\begin{equation}
  \label{eq:pij}
p_{ij} \coloneq K\bigl(M(X_i,|X_i - X_j|)\bigr),
\end{equation}
where $K \colon [0,1] \to \R^+$
and, for $r>0$, the function
\begin{equation}
  \label{eq:Mpart}
  M(X_i,r) \coloneq \frac 1N \sum_{k=1}^N \fcr\{|X_k - X_i| \le  r\}
\end{equation}
counts the number of agents at distance
less than or equal to $r$ from $X_i$, normalized with $N$.
Note that in this case $p_{ij}$ is a stepwise function of the positions
of all the agents.
In the sequel we assume that $K$ is a positive decreasing
function, Lipschitz continuous, and
such that $\int_0^1 K(z) \, \de z= \gamma$.

In the mean-field limit $N\to +\infty$, the one-agent distribution
function $f_t = f(t,x,v)$ is expected to verify the equation 
\begin{equation}
\label{eq:principale}
\partial_t f_t + v \cdot \nabla_{x} f_t+\nabla_v \cdot \left(
W[S f_t,f_t](x,v)
f_t\right)=0,
\end{equation}
where $Sf_t(x)\coloneq\int f_t(x,v) \de v$ is the spatial distribution
and where, given a probability density $f$ in $\R^d\times \R^d$
and a probability density $\rho$ in $\R^d$, 
\begin{equation}
\label{eq:interazione}
W[\rho, f](x,v)\coloneq \int K\left(M[\rho](x,|x-y|)\right)\,(w-v)f(y,w) \, \de y \, \de w,
\end{equation}
with
\begin{equation}
  \label{eq:M}
  M[\rho](x,r)\coloneq\int_{|x'-x|\le r} \rho(x') \, \de x'.
\end{equation}

A weak formulation of this equation is given
requiring that the solution $f_t$ fulfills
$$
\int \alpha(x,v) \de f_t(x, v) =
\int \alpha\left(X_t(x,v),V_t(x,v)\right) \de f_0(x,v)
$$
for any $\alpha \in C_b(\R^d\times \R^d)$, 
where $f_0$ is the initial probability measure 
and $(X_t(x,v),V_t(x,v))$ is the flow defined by
\begin{equation}
  \label{eq:flusso}
  \left\{
    \begin{aligned}
      &\dot X_t(x,v)=V_t(x,v)\\ 
      &\dot V_t(t,x,v)=W[Sf_t,f_t](X_t(x,v,),V_t(x,v))  \\
      & X_0(x,v) = x,\ \ V_0(x,v) = v.\\
    \end{aligned}
  \right.
\end{equation}
In other words, $f_t$ is the push-forward
of $f_0$ along the flow generated by the velocity field, determined
by  $f_t$ itself.

It is  easy to verify that the empirical measure
\[
  \mu^N_t \coloneq
  \frac{1}{N}
  \sum_{i=1}^N \delta_{X_i^N(t)} \, \delta_{V_i^N(t)}
\]
associated with the solution of \eqref{eq:CS}, \eqref{eq:pij} and
\eqref{eq:Mpart}
is a weak solution of \eqref{eq:principale}.
Namely,
$M[S\mu^N_t](X,r)$ is exactly 
$M(X,r)$  defined in \eqref{eq:Mpart}
(from now on we use the more complete notation
$M[S\mu^N_t](X,r)$). Thus, we can rewrite the agent evolution in \eqref{eq:CS}
as 
\begin{equation}
  \label{eq:particelle}
  \left\{
    \begin{aligned}
      &\dot X_i^N(t)=V_i^N(t)\\
      &\dot V_i^N(t)= W[S\mu^N_t,\mu^N_t](X^N_i(t),V^N_i(t)).
    \end{aligned}
  \right.
\end{equation}

In the Dobrushin approach to the mean-field limit,
the result is achieved from 
this fact and from the weak continuity, w.r.t the initial datum,
of the weak solutions of \eqref{eq:principale}. 
We cannot use this approach in presence of topological
interaction, since in general
the solutions of 
\eqref{eq:flusso} are not weakly continuous
w.r.t the initial datum (see Section \ref{sez:agenti}).
We can overcome this difficulty if the solution 
of \eqref{eq:principale} has a bounded density.
To obtain our result, we adapt the ideas used in 
\cite{trocheris} for the derivation
of the one-dimensional Vlasov equation in presence
of discontinuity of the force. In particular we prove that

\noindent
[Thm. \ref{teo:agenti}]
the $N$-particle dynamics is well defined, 
except for a set of measure zero;

\noindent
[Thm. \ref{teo:emfl}] if $f_0$ is bounded, there exists a unique
weak solution $f_t$ of the topological Cucker-Smale equation, which is
bounded;

\noindent
[Thm. \ref{teo:mfl}] if $\mu^N_t$ solves \eqref{eq:particelle} and
$\mu^N_0\debole \mu_0$, then $\mu^N_t\debole f_t$.

\vskip,3cm
We divide the work as follows: in Section \ref{sez:metriche} we
discuss some properties of the ``discrepancy distance'',
the main tool for dealing with topological interactions. 
In Section \ref{sez:agenti} 
we discuss existence, uniqueness and regularity of the agent dynamics
\eqref{eq:particelle}, proving Thm. \ref{teo:agenti}.
In Section \ref{sez:teo-linfinito}
we discuss existence, uniqueness and regularity of the weak solutions
of the mean-field equation
\eqref{eq:principale} with bounded initial datum, proving Thm. \ref{teo:emfl}.
In Section \ref{sez:limite} we prove Thm. \eqref{teo:mfl}.

\section{Distances and weak convergence}
\label{sez:metriche}

We recall that the 1-Wasserstein  distance $\wass $ of two 
probability measures $\rho_1$ and $\rho_2$ on $\R^d$ can be defined by duality
with Lipschitz functions:
$$
\begin{aligned}
\wass (\rho_1,\rho_2) &= \sup_{\phi \in C_b(\R^d), \text{Lip}(\phi) \le  1}
\int \phi (\de\rho_1 - \de\rho_2)\\
&=
\sup_{\phi \in C^1_b(\R^d), \|\grad \phi\|_{\infty} \le  1}
\int \phi (\de\rho_1 - \de\rho_2),
\end{aligned}
$$
where $\text{Lip}(\phi)$ is the Lipschitz constant of $\phi$.

The counter of the number of particles in \eqref{eq:M}
is not continuous w.r.t. $\wass$, so we work with the
weaker topology induced by another distance, the discrepancy,
defined as 
$$\disc (\rho_1,\rho_2) \coloneq
\sup_{x,r>0} \left | \int_{B_r(x)} \de\rho_1 -
  \int_{B_r(x)} \de\rho_2\right|.$$
Here and after, we denote by $B_r(x)$ the closed ball of center
$x$ and radius $r$ in $\R^d$.
In the sequel, 
we also indicate by $B_R$ the closed ball $B_R(0)$.
The discrepancy distance is mostly used in one dimension
to quantify the uniformity of sequence of points
(see \cite{KN,GS}), but
its multidimensional version is cited in \cite{Neuzert}, in
the contest of kinetic limits.


By definition, it holds the following proposition.
\begin{Proposition}[Lipschitzianity of $M$ w.r.t. $\disc$]
  \label{prop:contM}
  Let $\rho_1 $ and $\rho_2$ be two probability measures on $\R^d$.
  Then, for any $x\in \R^d$ and $r>0$,
  $$|M[\rho_1](x,r) - M[\rho_2](x,r)|\le \disc(\rho_1,\rho_2).$$
\end{Proposition}

We can also define $\disc$ in terms of regular
functions. Let $X$ be the subset of $C_b^1([0,+\infty);\R)$,
and
define 
$$\|\phi\|_X \coloneq \int_0^{+\infty} |\phi'(r)|\de r.$$
Then
$$\disc (\rho_1,\rho_2) = \sup_{\phi \in X: \, \|\phi\|_X \le 1}
\sup_x
\int \phi\bigl(|x -y|\bigr)\bigl(\de\rho_1(y)-\de\rho_2(y) \bigr).$$
This assertion is an easy consequence of the following lemma.
\begin{lemma}
  \label{lemma:equiv}
  Let $g_1$ and $g_2$ be two probability measures on $[0,+\infty)$.
  Then
  \begin{equation}
    \label{eq:ugu-r}
    \sup_{r\ge 0} \left|\int_{[0,r]}\de g_1-\int_{[0,r]}\de g_2\right| =
    \sup_{\phi\in X:\, \|\phi\|_X \le 1} \int_0^{+\infty} \phi \left(\de g_1 -\de g_2\right).
  \end{equation}
\end{lemma}

\begin{proof}
  Fix $r>0$, there exists
  $\phi_{r,\eps}\in X$ with $\|\phi_{r,\eps}\|_X = 1$
  and such that $\phi_{r,\eps} (s)=1$ if $0 \le s \le r$ and
  $\phi_{r,\eps}(s) = 0$ if $s\ge r+\eps$. For any measure $g$, 
  $$\lim_{\eps \to 0}\int_0^{+\infty} \bigl(\phi_{r,\eps}(s) - \fcr\{s\in [0,r]\}\bigr) \de g(s)=0,$$
  then
  $$
  \int_{[0,r]} (\de g_1- \de g_2) = \lim_{\eps \to 0} \int_0^{+\infty} \phi_{r,\eps}
  (\de g_1 - \de g_2) \le
  \sup_{\phi\in X:\, \|\phi\|_X \le 1} \int_0^{+\infty} \phi (\de g_1 - \de g_2).$$
  To prove the opposite  inequality, we denote by $G_1$ and $G_2$
  the distribution functions
  of $g_1$ and $g_2$:
  $$G_i (r) \coloneq \int_{[0,r]} \de g_i.$$
  Then, integrating by parts,
  $$\int_0^{+\infty} \phi (\de g_1 - \de g_2) = -
  \int_0^{+\infty} \phi'(r) \bigl(G_1(r) - G_2(r)\bigr)\de r \le
  \|\phi\|_X \|G_1-G_2\|_{\infty}.$$
  We conclude the proof by noticing that
  $\|G_1-G_2\|_{\infty}$ is exactly the left-hand-side of
  \eqref{eq:ugu-r}.
\end{proof}


  For our purposes,
  we  need the equivalence of
  $\disc$ and $\wass$ in the case in which one of the two measures
  has bounded density.
  We note  that in the general case
  the equivalence is false, as can
  be easily checked by considering two Dirac measures
  $\delta_{x_1}$ and $\delta_{x_2}$:
  $\wass$ vanishes when $|x_1 - x_2|\to 0$, while $\disc$ is one
  whenever $x_1 \neq x_2$.
  Nevertheless, 
  using the covering principles as in \cite{Bes},
  for measures on a compact set,
  it can be proved the continuity of the Wasserstein distance $\wass$ 
  w.r.t. the discrepancy distance $\disc$.
  For the sake of completeness, 
  we give a proof in the appendix, 
  although this property is not really necessary for our results.


\vskip.3cm

In the sequel, in the definition of $\disc$ we choose  functions in 
$\phi\in C([0,+\infty),\R)$, with first derivative
continuous up to a finite number of jumps. With abuse of notation,
we keep calling this set of functions $X$.
Let us expose some technical properties.

Given $\phi\in X$, we define some useful regularizations, 
$\phi^\pm$, $\phi_\eps$ and $\psi_\eps$, with $\eps>0$, as follows.
Denoting by $\tilde \phi$ the function
$$
\tilde \phi (r) \coloneq \int_0^r |\phi'(s)|\de s,$$
we define 
$$\phi^\pm (r)\coloneq \left\{
  \begin{aligned}
    &\frac 12 (\tilde \phi(r) \pm \phi(r)),\  &\text{ if }r\ge 0,\\
    &\pm \frac 12 \phi(0), &\text{ if }r<0,
  \end{aligned}
\right.
$$
and
\begin{equation}
	\label{phieps}
	\phi_\eps(r)  \coloneq \phi^+ (r+\eps) - \phi^-(r-\eps).
\end{equation}
Finally, fixed a regular mollifier  $\eta$ supported in $(0,1)$,
we define 
\begin{equation}
	\label{psieps}
\psi_\eps(r) \coloneq \int_0^\eps \eta_\eps(s) \phi^+(r+s) \de s -
\int_0^\eps \eta_\eps(s) \phi^-(r-s) \de s.
\end{equation}
where $\eta_\eps(s)\coloneq\eps^{-1}\eta(s/\eps)$.

We summarize the properties of these regularizations
in the following lemma, where we indicate with $c$ any constat
which does not depends on $\phi$ and $\eps$.
\begin{lemma}{\phantom{a}}
  \label{lemma:phipm}
  \begin{enumerate}[\it i)]
  \item 
    $\phi^\pm$ are not decreasing. Moreover
    \begin{equation}
    	\label{stima1}
    	\int_0^{+\infty} (\phi^\pm)'(r)\, \de r \le  \|\phi\|_X
    \end{equation}
    and 
    $\phi (r) = \phi^+ (r) - \phi^-(r)$
    for $r\ge 0$.
  \item
    $\phi_\eps\in X$, 
    $\phi(r) \le \phi_\eps(r)$
    and
    \begin{equation}
    	\label{stima2}
   \int_0^{+\infty} \bigl(\phi_\eps(r) - \phi(r)\bigr)\de r \le 2\eps \|\phi\|_X.
    \end{equation}
  \item  $\psi_\eps(r)\ge \phi(r)$. Moreover $\psi_\eps$ is a $C^1$
    function in $X$,
    \begin{equation}
    	\label{stima3}
    	\|(\psi_\eps)'\|_{\infty}
    	\le \frac 2\eps \|\eta\|_\infty \|\phi\|_X
    \end{equation}
    and
    \begin{equation}
    	\label{stima4}
    	\int_0^{+\infty}|\psi_\eps(r) - \phi(r)|\de r \le c\eps \|\phi\|_X.
      \end{equation}
  \end{enumerate}
\end{lemma}
\begin{proof}
  The proof is elementary, we only describe how to get the bounds in 
  {\it ii)} and {\it iii)}.
Since $\phi = \phi^+-\phi^-$,
we rewrite the l.h.s. of \eqref{stima2} as 
$$
\begin{aligned}
  &\int_0^{+\infty}\bigl(\phi^+(r+\eps) - \phi^+(r)\bigr) +
 \bigl(\phi^-(r)- \phi^-(r-\eps)\bigr) \de r\\
  &=   \int_0^{+\infty} \left(\int_0^\eps \left( (\phi^+)'(r+\xi) +
    (\phi^-)'(r-\xi) \right) \de \xi \right) \de r \le 2\eps \|\phi\|_X.
\end{aligned}
$$ 
The estimate in \eqref{stima3} is immediate while, regarding \eqref{stima4},
we rewrite 
$\psi_\eps(r) - \phi(r)$ as  
$$
\begin{aligned}
	&\int_0^1 \eta(s)\bigl(
	\phi^+(r+\eps s) - \phi^+(r) + \phi^-(r) - \phi^-(r-\eps s)\bigr)\de s \\
	 &=\eps \int_0^1 s\eta(s) \left(\int_0^1 (\phi^+)'(r+\eps s\xi)\de \xi+\int_0^1 (\phi^-)'(r-\eps s\xi)\de \xi \right) \de s.
\end{aligned}
$$
We conclude by integrating in $r$, switching the order of integration and using \eqref{stima1}.
\end{proof}

Now we can prove the following proposition.
\begin{Proposition}{}
  \label{prop:dw1}
  Let $\rho $ and $\nu$ be two probability measures on $\R^d$
  with support in a ball $B_R$ and such that $\rho \in L^\infty(\R^d)$.
  Then
  $$\disc (\nu, \rho) \le C(\|\rho\|_{\infty}, R) \sqrt{\wass (\nu,\rho)},$$ 
  where $C$ is a constant that depends on the dimension $d$, as well as on $\|\rho\|_{\infty}$ and on $R$.
\end{Proposition}
\begin{proof}
  Let $\phi$ be in $X$ and consider $\psi_\eps$ as in \eqref{psieps}.
  Fixed $x \in \R^d$, let $\Phi$ and $\Psi_\eps$ be the functions
  $$\Phi(y) \coloneq \phi(|x-y|)\ \ \text{and} \ \ \Psi_\eps(y) \coloneq \psi_\eps(|x-y|).$$
  Then, from {\it iii)} of Lemma \ref{lemma:phipm},
  $$\int \Phi \de \nu  - \int \Phi \de \rho \le
  \int \Psi_\eps \de \nu - \int \Phi \de \rho =
  \int \Psi_\eps \de(\nu - \rho ) + \int (\Psi_\eps -\Phi) \de\rho.$$
  From \eqref{stima3} of Lemma \ref{lemma:phipm},
  the first term is bounded
  by
  $\frac c{\eps} \|\phi\|_X \wass (\nu,\rho)$.
  Regarding the second term, denoting by $\sigma_r$ the uniform measure
  on $\pa B_r(x)$,  we have 
  \begin{equation}
  	\label{stimarho}
  \begin{aligned}
    \int (\Psi_\eps -\Phi) \de\rho
    &\le \| \rho\|_\infty \int_0^{+\infty}\de r
    \left(\psi_\eps(r)-\phi(r) \right)
    \int_{\pa B_r(x)} \fcr\{z\in B_R\} \sigma(\de z) 
    \\
  &\le c\eps R^{d-1} \|\phi\|_X \|\rho\|_{\infty},
  \end{aligned}
  \end{equation}
  where in the last inequality we have used \eqref{stima4}.
  Optimizing on $\eps$ and passing to the supremum in $\phi$,
  we get the proof.
\end{proof}

Note that if $\mu^N$ is an empirical measure and $\nu$ a probability measure
that does not give mass to the atoms of $\mu^N$,
$\disc(\mu^N,\rho) \ge 1/N$.
With this constraint, the discrepancy between two empirical measures
is ``small'' if the measures are close in the
sense specified in the following proposition.

\begin{Proposition}
  \label{prop:Dmunu}
  Let
  $$\mu^N = \frac 1N \sum_{i=1}^N \delta_{x_i}\ \ \text{and} \ \
  \nu^N =\frac 1N \sum_{i=1}^N \delta_{y_i}$$
  be two empirical measures on $\R^d$ and take $\delta>0$ such that $|x_i - y_i|\le \delta$ for all $i=1,\dots,N$. Then, for any probability measure $\rho \in L^{\infty}(\R^d)$
  supported on a ball $B_R$, 
  $$\disc (\mu^N,\nu^N) \le cR^{d-1}\delta\|\rho\|_{\infty} +
  c \disc (\mu^N,\rho).$$
\end{Proposition}

\begin{proof}
  Given $\phi\in X$ with $\|\phi\|_X\le 1$,
  we construct $\phi_\delta$ as in \eqref{phieps} and, fixed $x \in \R^d$,
  we consider $\Phi(y) \coloneq \phi(|x-y|)$, $\Phi_\delta(y) \coloneq
  \phi_\delta(|x-y|)$.

  Since $|x-x_i|-\delta \le |x - y_i| \le |x-x_i| + \delta$, we have that
  $$\Phi(y_i) = \phi^+(|x-y_i|) - \phi^-(|x-y_i|) \le \Phi_\delta(x_i).$$
  Then 
  $$\int \Phi \de (\nu^N-\mu^N) \le \int (\Phi_\delta - \Phi)\de \mu^N =
  \int (\Phi_\delta - \Phi)\de(\mu^N -\rho) + \int (\Phi_\delta - \Phi)\de\rho.$$
  Since $(\phi_\delta - \phi) \in X$, the first term is bounded by
  $c\disc (\mu^N,\rho)$. Using \eqref{stima2} and reasoning as in \eqref{stimarho} we estimate the second term with
  $c\delta R^{d-1}\|\rho\|_{\infty}.$
\end{proof}

\section{Agent dynamics}
\label{sez:agenti}

One of the difficulties in handling \eqref{eq:particelle} is that
the dynamic is not continuous w.r.t the initial datum.
For instance, consider three agents $\{X_i\}_{i=1}^3$ on a line, such that
\begin{equation}
	\label{eq:datoini}
	\begin{array}{lll}
		X_1(0) = -1,  &X_2(0) = \eps,  &X_3(0) = 1, \\ 
		\,V_1(0)=-1, &\,V_2(0) =0, &\,V_3(0) = 1,
	\end{array}
\end{equation}
with $\eps \in (-1,1)\setminus\{0\}$.
Then $p_{i,j}=M(X_i,|X_i - X_j|)$ takes the values $1/3,2/3, 1$.
Suppose for simplicity that $K(2/3)=3$ and $K(1)=0$, then
the equations for $V_1$ and $V_3$ read as
$$
\left\{\begin{aligned}
  &\dot V_1(t)  = V_2(t) - V_1(t)\\
  &\dot V_3(t)  = V_2(t) - V_3(t),
\end{aligned}
\right.
$$
while
$$
\dot V_2(t) = 
 \begin{cases}
   V_3(t) - V_2(t) & \text{ if } \eps \in (0,1)\\
   V_1(t) - V_2(t) & \text{ if } \eps \in (-1,0).
  \end{cases}
$$
It follows that 
$$
\left\{
\begin{aligned}
  &V_1(t) = -(1+\e^{-2t})/2\\
  &V_2(t) = -(1-\e^{-2t})/2\\
  &V_3(t) = (-1+4\e^{-t} - \e^{-2t})/2\\
\end{aligned}
\right.
$$
if $\eps \in (-1,0)$,
while
$$
\left\{
\begin{aligned}
  &V_1(t) = -(-1+4\e^{-t} - \e^{-2t})/2\\
  &V_2(t) = (1-\e^{-2t})/2\\
  &V_3(t) = (1+\e^{-2t})/2\\
\end{aligned}
\right.
$$
if $\eps \in (0,1)$, so that $\{X_i(t),V_i(t)\}_{i=1}^3$ is discontinuous in $\eps=0$.
Note that the discontinuity of the trajectories in the phase space is
easily translated in the weak discontinuity of the empirical measure
at time $t$, 
w.r.t the initial measure.


This discontinuity reflects the fact that,
for data as in \eqref{eq:datoini} with $\eps=0$, 
there is not a unique way to define
the dynamics.
Nevertheless, we can prove that the system \eqref{eq:particelle}
is well-posed for almost all initial data.
To do so, let us define some subsets of the phase space 
$$\left\{(X,V) \coloneq (x_1,\dots, x_N, v_1,\dots, v_N)\in
  \R^{Nd}\times \R^{Nd}\right\},$$
where $d\ge 1$ is the dimension of the configuration space of the
agents.

\begin{Definition}{\phantom{a}}
  \begin{itemize}
  \item[$\Cal R$] is 
    the set of ``the regular points'', i.e. 
    the set of points $(X,V)$ such that for each triad of
    different indices it holds that
    $|x_i-x_k| \neq |x_j-x_k|$.
  \item[$\Cal S$] is the ``iso-rank'' manifold, i.e. 
    the set of points $(X,V)$ such that
    there exists a triad of different
    indices $i, j, k$ for which 
    $|x_i-x_k| =|x_j-x_k|$, i.e. the agents $i$ and $j$ have the same
    rank with respect to the agent $k$.
  \item[$\Cal S_r$] is the set of the ``regular points''
    of the iso-rank manifold,
    i.e.. the subset of points
    $(X,V)\in \Cal S$ such that if 
    $|x_i-x_k| =|x_j-x_k|$ then $x_i$, $x_j$, $x_k$ are different
    and 
    $(v_i-v_k) \cdot \hat n_{ik} \neq (v_j-v_k) \cdot \hat n_{jk}$,
    where $\hat n_{ab} \coloneq (x_a-x_b)/|x_a-x_b|$.
  \end{itemize}
\end{Definition}
We can define the dynamics locally in time,
not only for initial data in  $\Cal R$, but also in $\Cal S_r$.
Namely, if initially
the agents $i$ and $j$ have the same rank with respect to the
agent $k$,
we can redefine the force exerted on the agent $k$ accordingly
to the velocities:
if $(v_i-v_k) \cdot \hat n_{ik} > (v_j-v_k) \cdot \hat n_{jk}$
we evaluate the rank as if $|x_i -x_k| > |x_j-x_k|$ for $t>0$
and as if $|x_i -x_k| < |x_j-x_k|$ for $t<0$.
In other words, the different speeds of change of the distances
among the agents allow the dynamics to leave $\Cal S$ instantaneously.

We discuss the existence of the dynamics, so redefined.
\begin{lemma}{}
  \label{lemma:esistenza-loc}
  If $(X,V) \in \Cal R \cup \Cal S_r$, there exists $\tau > 0$
  such that the system \eqref{eq:particelle} has a
  unique solution for $t\in (-\tau,\tau)$, with initial
  datum $(X,V)$. Moreover the solution is locally Lipschitz in $t$
  and in $(X,V)$.
\end{lemma}
\noindent
We omit the proof.

\vskip.3cm
In $\Cal R$ the solution is regular, so we can compute
the determinant of the  Jacobian of the flow $J(t)\equiv J(X,V,t)$.
It verifies the equation
\begin{equation}
\label{Jacobian}
\opde tJ(t)  = - \left( \frac dN \sum_{i,j: i\neq j} p_{ij} \right)J(t)  = -dN\gamma_N J(t),
\end{equation}
where
$$\gamma_N \coloneq \frac 1N \sum_{n=2}^{N} K\left(n/N \right).$$
Thus, volumes of the phase space are shrunk in time at a constant rate, therefore
their measure cannot vanish in finite time.
This implies the following fact,  of which we omit the proof. 
\begin{lemma}{}
  \label{lemma:misura0}
  The subset of initial data $(X,V) \in \Cal R$
  such that the trajectory,  at a first time in the future or in the past, intersects
  $\Cal S \setminus \Cal S_r$, has Lebesgue measure zero. Namely, $\Cal S \setminus \Cal S_r$ has dimension
  $2Nd - 2$.
\end{lemma}

This lemma guarantees that, except for a subset of Lebesgue measure zero,
we can prolong the dynamics with initial data
in $\Cal R$
also after a crossing in $\Cal S$.
To define the dynamics for all times, we need to control
the number of crossings.
\begin{lemma}{}
  \label{lemma:misura0-bis}
  The subset of initial data $(X,V) \in \Cal R$
  such that the trajectory intersect 
  $\Cal S_r$ infinitely many times in finite time, has Lebesgue measure zero.
\end{lemma}
\begin{proof}
Fix $T>0$ and suppose to take $(X,V) \in \Cal R$ such that the solution $\left(X^N(t),V^N(t)\right)=(X_1(t), \dots, X_N(t), V_1(t), \dots, V_N(t))$ with initial data $(X,V)$ intersects $\Cal S_r$ a finite number of times
in $[0,T-\eps)$ and infinitely many times in $[0,T)$.
The number of particles is finite, so we can assume that
there exists a triad of indices such that
$|X_i-X_k|=|X_j-X_k|$ infinitely many times.
Since the velocities $V_i$ are bounded by a constant, as follows by
simple considerations  (see also Thm. \ref{teo:agenti}), 
from the equation we have that $|X_i-X_k|$ and
$|X_j-X_k|$ are $C^1$ functions, with time derivatives
uniformly Lipschitz, if $|X_i-X_k|$ and $|X_j-X_k|$ remain
far from $0$.
Then, as $t\to T$, either $|X_i-X_k|\to 0$ or
$(V_i-V_k) \cdot \hat n_{ik}$ and $(V_j-V_k) \cdot \hat n_{jk}$
converge to the same limit.
In both the cases, the trajectory reaches $\Cal S$
at a point that is not in $\Cal S_r$.
We conclude the proof observing that
the intial point with these properties lives in
a subset of dimension $2Nd-1$.
\end{proof}

From these lemmas and other few considerations, we obtain the following
theorem.
\begin{Theorem}{}
  \label{teo:agenti}
  Except for a set of measure zero, given $(X,V) \in \R^{Nd}\times \R^{Nd}$,
  there exists a unique global solution  
  $$\left(X^N(t,X,V),V^N(t,X,V)\right)
  \in C^1(\R^+,\R^{2dN})\times C(\R^+,\R^{2dN})$$
  with initial
  datum $(X,V)$.

  Moreover, given  $R_x>0$ and $R_v>0$,
  we have that
  $$|X_i(t)| \le R_x+tR_v, \ |V_i(t)| \le R_v$$
  for any $i$, if $|x_i|\le R_x$ and $|v_i|\le R_v$.
  Therefore $V_i(t,X,V)$  has Lipschitz constant
  bounded by $2R_v K(0)$.
\end{Theorem}
\begin{proof}
	The proof follows easily from Lemma \eqref{lemma:esistenza-loc}, Lemma \eqref{lemma:misura0} and Lemma \eqref{lemma:misura0-bis}. 
	
	The {\it a-priori} bound on the support follows from \eqref{Jacobian} and by noticing that 
	$$\opde t |V_i(t)|^2 = - 2\sum_{j \neq i} p_{ij} \left( |V_i(t)|^2 - V_i(t) \cdot V_j(t)\right)$$
	is null or negative if $|V_i|^2$ is maximum in $i$.
	
\end{proof}

\section{The mean-field equation in $L^{\infty}$}
\label{sez:teo-linfinito}

In this section we show how to get an existence and uniqueness result 
for bounded weak solutions of equation \eqref{eq:principale}. 
We start by stating some elementary facts.
\begin{lemma}
\label{lemma:stimemasse}
Let  $\rho \in  L^{\infty}(\R^d)$ be a probability density.
\begin{enumerate}[i)]
\item Given $r_1, r_2>0$, 
\[
\left|M[\rho](x, r_1) - M[\rho](x,r_2)\right| \le c\|\rho\|_{\infty} \left|r_1^d-r_2^d\right|.
\]
\item Given $x_1, x_2 \in \R^d$ and $r>0$,
\[
\left|M[\rho](x_1, r)- M[\rho](x_2, r)\right| \le c \|\rho\|_{\infty} r^{d-1}|x_1-x_2|.
\]
%
\end{enumerate}
\end{lemma}

\begin{proof}
  The proof of the first assertion is immediate.
  For the second, we use the following splitting
  $$
  \begin{aligned}
  \fcr\{|x_1-y|<r\}-\fcr\{|x_2-y|<r\}&=\fcr\{|x_1-y|<r\}\fcr\{|x_2-y|\ge r\} \\
  &- \fcr\{|x_2-y|<r\}\fcr\{|x_1-y|\ge r\}
  \end{aligned}
  $$
  and we note that, if $|x_1-x_2|\ge r$,
  $$\int_{|x_1-y|<r}\fcr\{|x_2-y |\ge r\} \de y \le c r^d \le c r^{d-1} |x_1 -x_2|,$$
  while, if $|x_1 - x_2| < r$,
  $$
  \begin{aligned}\int_{|x_1-y|<r}&\fcr\{|x_2-y |\ge r\}\de y \le
  \int \fcr \left\{r-|x_1-x_2| < |x_1 -y| < r\right\}\de y\\ &= c
  r^{d}\left(1 - \left(1-|x_1-x_2|/r\right)^d\right) \le cd r^{d-1} |x_1-x_2|.
  \end{aligned}$$
\end{proof}
In the following, we denote by $\mathcal{B}_r$ the closed ball of center $0$ and radius $r$ in $L^\infty(\R^d \times \R^d)$ and by 
$
C_w \left([0,+\infty); L^{\infty}(\R^d \times \R^d )\right)
$
the set of families of bounded probability densities $\{f_t\}_{t \ge 0}$ which are weakly continuous in time in the sense of measures. 

\begin{lemma}
  \label{lemma:flusso}  
  Let $\{f_t\}_{t \ge 0}$ be a family of probability densities
  such that $\{f_t\} \in C_w \left([0,+\infty); \mathcal{B}_{r(t)}\right)$, with $r(t)$ a continuous nondecreasing function.
  Suppose that 
  \begin{equation}
  	\label{support}
  		\text{supp}(f_t)\subset B_{R_x(t)}\times B_{R_v(t)}, 
  \end{equation}
where $R_v(t)$ and $R_x(t)$ are two
  continuous non decreasing functions.
  Then, for any initial datum $(x,v)\in \R^{d}\times \R^d$, 
  there exists a unique global solution of \eqref{eq:flusso}.
\end{lemma}

\begin{proof}
  From the classical Cauchy-Lipschitz theory, we only have to verify that $W[Sf_t,f_t](x,v)$ is bounded
  on compact sets,  locally Lipschitz
  and continuous in $t$.

  Recalling \ref{eq:interazione}, the boundness on compact sets follows from
  $$|W[Sf_t,f_t](x,v)| \le \|K\|_{\infty} \left(R_v(t)+ |v|\right).$$
  Since from {\it i)} and {\it ii)} of Lemma \ref{lemma:stimemasse}
  $$
  \begin{aligned}
   |M[Sf_t](x_1,|x_1-y|) &- M[Sf_t](x_2,|x_2-y|)| \\
    &\le c\|Sf_t\|_{\infty}( |x_1| + |x_2|+ |y|)^{d-1} |x_1-x_2|
  \end{aligned},$$
  we have that, if $(x_1,v_1)$ and $(x_2, v_2)$ belong to a compact subset of
  $\R^d\times \R^d$,
  $$|W[Sf_t,f_t](x_1,v_1) - W[Sf_t,f_t](x_2,v_2)|
  \le C (|x_1-x_2|+|v_1 -v_2|),$$
  where $C$ depends on $R_x, R_v$ and on the diameter of the compact set.

  In order to prove that $W[Sf_t,f_t](x,v)$ is continuous in $t$, we first observe that
  $\wass (Sf_t,Sf_s) \le \wass (f_t,f_s)$
  and that, from the Lipschitzianity of $K$ and Propositions \ref{prop:contM} and
  \ref{prop:dw1},
  $K\left(M([Sf_t](x,|x-y|)\right)$ is continuous in $t$. Since $K(M([Sf_t](x,|x-y|))$
  is Lipschitz in $y$, also
  $$\int K\left(M([Sf_t](x,|x-y|)\right) (v-w) \left(f_t(y,w) -f_s(y,w)\right)\de y \de w$$
  vanishes when $\wass (f_t,f_s)\to 0$.
\end{proof}

Now we can prove the main theorem of this section.
\begin{Theorem}
  \label{teo:emfl}
  Let $f_0(x,v)\in L^\infty(\R^d \times \R^d)$ be a
  probability density such that $\text{supp}(f_0)
  \subset B_{R_x} \times B_{R_v}$. Given $T>0$,
  there exists a unique weak solution
  $f\in C_w\left([0,T];L^{\infty}(\R^d\times \R^d)\right)$
  of the topological Cucker-Smale equation. 
  Moreover 
    \begin{equation}
    \label{eq:supporto}
    \text{supp}(f_t)\subset B_{R_x+tR_v}\times B_{R_v}.
    \end{equation}
\end{Theorem}
\begin{proof}
  We first note that, if the solution exists, 
  \eqref{eq:supporto}  follows
  from an argument similar to the one used 
  in the discrete case (see Theorem \ref{teo:agenti}).

  We now prove the existence.
  As in Lemma \ref{lemma:flusso}, consider a family of probability densities $\{g_t\}_{t \ge 0}
  \in 
  C_w\left([0,T];\mathcal{B}_{M}\right),
  $
  with $M\coloneq \|f_0\|_{\infty} \e^{d\gamma T}$ and such that \eqref{support}
  holds with $R_x(t) = R_x + t R_v$ and $R_v(t)=R_v$.
  The push-forward of $f_0$ along the flow generated
  by $g_t$, denoted by $\tilde g_t$, is weakly continuous
  in $t$, uniformly in $g_t$, with $t\in [0,T]$.
  Moreover, the determinant 
  of the Jacobian of the flow $J(t)=J(t,x,v)$
  verifies
  $$\opde tJ(t)  = -J(t) d \gamma.$$
  So the push-forward $\tilde g_t$
  is bounded by $\|f_0\|_{\infty} \e^{d \gamma t}$.

  With a standard construction we can prove that,
  for $T$ sufficiently small,
  the map $  \{g_t\}\mapsto \{\tilde g_t\}$  is a contraction in $C_w\left([0,T];\mathcal{B}_{M}\right)$,
  with the distance defined by the supremum on time of the Wasserstein distance;
  in this way we prove
  local existence and uniqueness. Using
  the a-priori estimate on the supremum and on the support,
  we get the global result.
\end{proof}

\section{The mean-field limit}
\label{sez:limite}

In this section we prove the main result regarding the mean-field
limit for the topological Cucker-Smale equation.
In the sequel,
$f_t$ is the fixed global solution of
eq. \eqref{eq:flusso} as in Theorem \ref{teo:emfl}, with initial
datum $f_0$, 
and $\mu^N_t$ is the  global solution 
of eq. \eqref{eq:particelle}
in the sense of Theorem \ref{teo:agenti},
with initial datum 
$$\mu_0^N = \frac 1N \sum_{i=0}^N \delta_{x_i} \delta_{v_i}.$$
We assume that $f_0$ and $\mu^N_0$ are supported in $B_{R_x}\times B_{R_v}$.
Fixed $T$,
we indicate
by $C(T)$ any constant that
depends only on $T$, $R_x$, $R_v$  and $\|f_0\|_{\infty}$.

In order to get the result, 
we compare the $N$-agent dynamics 
with the ``intermediate'' dynamics  given by 
\begin{equation}
  \label{eq:flussoeffe}
  \begin{cases}
    \dot X^f_i(t)=V^f_i(t) \\
    \dot V^f_i(t)= W[Sf_t,\nu^N_t](X^f_i,V_i^f),
  \end{cases}
\end{equation}
where
$$\nu^N_t \coloneq \frac 1N \sum_{k=1}^N \delta_{X^f_k(t)} \, 
\delta_{V^f_k(t)}$$
is the empirical measure. The initial datum 
is $\nu^N_0 = \mu^N_0$,
i.e. 
$$\{(X^f_i(0),V^f_i(0))\}_{i=1}^N = \{(x_i,v_i)\}_{i=1}^N.$$

\begin{Proposition}
  \label{lemma:munu}
  Given $T>0$, it holds that
  \begin{enumerate}[i)]
  \item  For $t \in [0, T]$, 
  \begin{equation}
  	\label{distanzaf}
  	  \wass (f_t,\nu_t^N) \le C(T) \wass(f_0,\mu^N_0).
  \end{equation}

  \item For $t \in [0, T]$, the distance 
    \[
      \delta(t) 
      \coloneq \max_{i=1, \dots, N}  \left(
        |X^f_i(t)-X^N_i(t)|+|V^f_i(t) - V^N_i(t)|\right)
    \]
    verifies
    \begin{equation}
      \label{eq:distanzaparticelle}
      \delta(t) \le C(T) \sqrt{\wass (f_0,\mu^N_0)}.
    \end{equation}
    \end{enumerate}
\end{Proposition}

\begin{proof}
  Since $f_t$ is bounded, 
  $K\left(M[Sf_t](x,|x-y|)\right)$ is locally Lipschitz in $x$ and $y$
  (see
  {\it i)} and {\it ii)} of Lemma \ref{lemma:stimemasse})
  and then $W[Sf_t,\nu](x,v)$ is weakly continuous in
  $\nu$ in the sense that
  $$\sup_{x,v} |
  W[S_f,\nu_1](x,v) -W[S_f,\nu_2](x,v)| \le C(T)
  \wass(\nu_1,\nu_2).$$
  It is straightforward to prove that
  the solution $\nu_t$ of the system
  $$
  \left\{
  \begin{aligned}
    &\dot X_t=V_t \\
    &\dot V_t= W[Sf_t,\nu_t](X_t,V_t) \\
    &\nu_t = \text{ push-forward of }\nu_0\text{ along the flow } (X_t,V_t)
  \end{aligned}
\right.
  $$
  is continuous in $\wass$ w.r.t. the initial datum $\nu_0$.
  Taking $\nu_0 =f_0$ and $\nu_0=\mu_0^N$ we get the proof of {\it i)}.

  In order to estimate $\delta(t)$, we need to evaluate,
  for $0\le s \le t$ and for $i=1, \dots, N$,
  the difference $|\dot V^f_i(s)- \dot V^N_i(s)|$ given by 
  $$|W[Sf_s,\nu_s^N](X^f_i,V^f_i) - W[S\mu^N_s,\mu_s^N](X^N_i,V^N_i)|.$$
We estimate this quantity with the sum of  three terms:
  $$\begin{aligned}
    (a)\ \  & |W[Sf_s,\nu_s^N](X^f_i,V^f_i) - W[Sf_s,\nu_s^N](X^N_i,V^N_i)|,\\
    (b)\ \  & |W[Sf_s,\nu_s^N](X^N_i,V^N_i) - W[Sf_s,\mu_s^N](X^N_i,V^N_i)|,\\
    (c)\ \  & |W[Sf_s,\mu_s^N](X^N_i,V^N_i) - W[S\mu^N_s,\mu_s^N](X^N_i,V^N_i)|.\\
  \end{aligned}
  $$
  Since $K\left(M[Sf_s](x,|x-y|)\right)$ is Lipschitz in $x$,
  from the definition of $W$ it is easy to prove that
  {\it (a)} is bounded by
  $$\left(c\,\text{Lip}(K)\|Sf_s\|_{\infty} R_x^{d-1}(s)R_v+ c\|K\|_{\infty}\right) \delta(s)$$
  and that {\it (b)} is estimated by
  $$c\,\text{Lip}(K)\|Sf_s\|_{\infty} R_x^{d-1}(s)R_v \delta(s).$$
  Note that $\|Sf_s\|_{\infty}\le cR_v^d \|f_s\|_{\infty}$.
  From  Proposition 
  \ref{prop:contM} we have that  {\it (c)} is bounded by
  $$ c\text{Lip}(K)R_v \disc (Sf_s,S\mu^N_s).$$
Since
  $$ \disc (Sf_s,S\mu^N_s) \le \disc (Sf_s,S\nu^N_s)+ \disc
  (S\nu^N_s,S\mu^N_s),$$
  by Proposition \ref{prop:Dmunu} with
  $\rho =Sf_s$, $\mu^N = S\nu^N_s$ and $\nu^N = S\mu^N_s$, we get
  $$\disc (S\nu^N_s,S\mu^N_s)\le c \delta(s) + c \disc (Sf_s, S\nu_s^N).$$
  Writing $\delta(t)$ in terms of the time integral of
  $\delta(s)$ and the difference of the interaction terms 
  and using 
  the Gronwall lemma, we readily get the estimate 
  $$\delta(t) \le C(T)\int_0^t \disc (Sf_s,S\nu^N_s) \de s,$$
  valid for $0\le t \le T$.
  We conclude the proof by using Proposition \ref{prop:dw1}, eq.
  \eqref{distanzaf}
  and the fact that $\wass (Sf_s,S\nu^N_s) \le \wass (f_s,\nu^N_s)$.
\end{proof}

\begin{Theorem}
	\label{teo:mfl}
	Fixed $T>0$, let $f_t$ be a solution of
	eq. \eqref{eq:flusso} as in Theorem \ref{teo:emfl} with initial
	datum $f_0$ and let
	$\mu^N_t$ 
	be a solution of eq. \eqref{eq:particelle}  
	in the sense of Theorem \ref{teo:agenti} 
	with initial datum $\mu_0^N$. 
	Then, for $0 \le t \le T$,
	$$\wass (f_t,\mu^N_t)\le C(T) \text{max} \left\{ 
	\wass (f_0,\mu^N_0), \sqrt{\wass (f_0,\mu^N_0)}
	\right\}.
	$$
	
\end{Theorem}
\begin{proof}
	By the triangular inequality,
	$$
	\wass(f_t,\mu^N_t) \le \wass(f_t,\nu^N_t)+\wass(\nu^N_t,\mu^N_t).
	$$
    From \eqref{distanzaf}, using that $\wass(\nu^N_t,\mu^N_t) \le \delta(t)$ and \eqref{eq:distanzaparticelle}, we get the thesis.
\end{proof}

 \section*{Appendix: Continuity of $\wass$ w.r.t. $\disc$}
In this appendix we prove 
the continuity of the Wasserstein distance $\wass$ 
w.r.t. the discrepancy distance $\disc$ for compactly supported measures.

Consider two probability measures $\mu$ and $\nu$,
both with support in the ball $B_R$ of $\R^d$.
Fix $\eps>0$ and consider a Lipschitz test function $ \phi$; it is
sufficient to consider $\phi$ with support of diameter less than
$cR$, so that $\| \phi \|_\infty \le cR$.
Given such $\phi$,  take $\delta_1>0$
such that $\text{Lip}(\phi) \delta_1 < \eps$.

By the Besicovitch covering principle  (see \cite{Bes}), 
there exist $N_\eps$ disjoint closed balls $\{B_i\}_{i=1}^{N_\eps}$ 
of radius at most $\delta_1$ such that
 $$
\mu\left ( \bigcup_{i=1}^{N_\eps} B_i \right) \ge 1- \eps.
$$ 
We estimate
$$
\int \phi \de(\mu - \nu)= \int_{\R^d  \setminus \bigcup B_i} \phi \de(\mu-\nu) + \int_{\bigcup B_i} \phi \de(\mu-\nu)\equiv A+B.
$$
We have that 
$$
\begin{aligned}
A &\le \| \phi \|_\infty \left ( \mu\left(\R^d \setminus\cup_{i=1}^{N_\eps} B_i\right)   + \nu\left(\R^d \setminus \cup_{i=1}^{N_\eps} B_i\right)\right ) \\
&\le  \| \phi \|_\infty \left (2\eps +N_\eps \disc(\mu, \nu) \right),
\end{aligned}
$$
while
$$
\begin{aligned}
B&\le \sum_{i=1}^{N_\eps}\int_{B_i} (\sup\phi-\inf\phi) \de\nu + \sum_{i=1}^{N_\eps} \int_{ B_i} \sup\phi \de(\mu-\nu)\\
&\le 2\text{Lip}(\phi)\delta_1+N_\eps \|\phi\|_\infty \disc(\mu, \nu).
\end{aligned}
$$
Hence we obtain
$$
\int \phi \de(\mu - \nu) \le cR\eps + cRN_\eps \disc(\mu,\nu).
$$
Taking $\disc(\mu,\nu)<\delta_2$  such that $N_\eps \delta_2 <\eps$, we get the thesis.


\end{document}